\documentclass{gtart}

\usepackage{graphicx, epsfig}
\usepackage{amsmath, amssymb, latexsym, euscript}
\usepackage{mathptmx, wrapfig}

\usepackage[inner=30mm, outer=30mm, textheight=225mm]{geometry}

\theoremstyle{plain}
\newtheorem{theorem}{Theorem}
\newtheorem{lemma}[theorem]{Lemma}
\newtheorem{proposition}[theorem]{Proposition}
\newtheorem{corollary}[theorem]{Corollary}

\theoremstyle{definition}

\newtheorem*{definition*}{Definition}

\theoremstyle{remark}
\newtheorem{remark}[theorem]{Remark}
\newtheorem{question}[theorem]{Question}
\newtheorem{example}[theorem]{Example}

\numberwithin{equation}{section}

\def\bkC{{\rm \kern.24em \vrule width.05em height1.4ex depth-.05ex \kern-.26em C}}

\def\bksC{{\rm \kern.24em \vrule width.05em height1ex depth-.05ex \kern-.26em C}}

\def\bkE{{\rm I\kern-.22em E}}
	
\def\bkN{{\rm I\kern-.22em N}}

\def\bkP{{\rm I\kern-.22em P}}
	
\def\bkH{{\rm I\kern-.22em H}}
			
\def\bkQ{{\rm \kern.24em \vrule width.05em height1.4ex depth-.05ex \kern-.26em Q}}
\def\Q{\bkQ}
\def\bkR{{\rm I\kern-.17em R}}

\def\R{\bkR}
\def\bkZ{{\rm Z\kern-.32em Z}}
\def\Z{\bkZ}
\def\bksZ{{\rm Z\kern-.22em Z}}

\def\Nzero{{\mathbb N}}
\def\smooth{}

\DeclareMathOperator{\Int}{int}
\DeclareMathOperator{\dd}{d}
\DeclareMathOperator{\dx}{dx}
\DeclareMathOperator{\im}{im}
\DeclareMathOperator{\rank}{rank}
\DeclareMathOperator{\hrank}{h-rank}
\DeclareMathOperator{\mrank}{m-rank}
\DeclareMathOperator{\grank}{g-rank}

\DeclareMathOperator{\Hom}{Hom}
\DeclareMathOperator{\cls}{cl}

\begin{document}

\title{Complexity functions on 1--dimensional cohomology}
\author{Daryl Cooper and Stephan Tillmann}

\begin{abstract} 
For a smooth, closed $n$--manifold $M$, we define an upper semi-continuous integer-valued  
complexity function on $H^1(M;{\mathbb R})$ using Morse theory. 
This measures how far an integral class is from being a fiber of a fibration. The fact complexity
 minimisers are open generalises Tischler's result on the openness of classes dual to fibrations.
We then use this to define a complexity function on 1--dimensional cohomology of a finitely presented group, which is constant on open rays from the origin and vanishes precisely on the geometric invariant due to Bieri, Neumann and Strebel.
\end{abstract}

\primaryclass{57M25, 57N10, 20E06}
\keywords{manifold, 1--form, finitely generated group, Bieri-Neumann-Strebel invariant, Thurston norm}
\makeshorttitle


\section{Introduction}

This work has its origin in a desire to exhibit extra structure in the non-fibred faces of the unit norm ball of the Thurston norm on first homology of a 3--manifold (see~\cite{Thu1986}). However, our constructions are more general. To keep the discussion brief, we restrict ourselves to closed manifolds and finitely presented groups in this note. Throughout this paper we regard $0$ as a natural number in ${\mathbb N}$ and manifolds are smooth.

Suppose $M$ is a connected \smooth  oriented closed $n$-manifold. The {\em Morse complexity function}  
$$\widetilde{m}:H^1(M;{\mathbb R})\longrightarrow \Nzero^{n+1}$$ 
is a vector, where the $k^{th}$ component counts the number of critical points of Morse index $k$ of a minimal closed Morse 1--form representing the cohomology class, with the minimum taken over all representatives using lexicographic order. 
Since $\widetilde{m}$ is constant on rays, it gives a function on the sphere
$$
m:S(M) = (H^1(M; \R) \setminus \{ 0\})\; /\; \R_+\longrightarrow\Nzero^{n+1}.
$$

Then $m(\psi)=\vec 0$ if and only if $\psi$ is represented by a 
non-singular closed 1--form. It follows from work of Tischler~\cite{Tisch} (see Fried~\cite[Theorem 1]{Fried}) 
that the set of cohomology classes of non-singular closed 1--forms is an open cone in $H^1(M; \R) \setminus \{ 0\},$ 
which is non-empty if and only if there is an integral class represented by a codimension--1 manifold that is a fiber of a fibration of $M$ over $S^1.$ 
 
\begin{theorem}\label{thm:intro main} Suppose $M$ is a closed \smooth orientable $n$-manifold.
For each $\vec m\in\Nzero^{n+1}$  there is an open subset
$U\subset S(M)$ such that $m(\psi)\le \vec m$ if and only if $\psi\in U$.
\end{theorem}

\begin{corollary} Suppose $M$ is a closed orientable $n$-manifold, and $S\subset M$ is a 2--sided connected 
non-separating codimension--1 submanifold, and $X$ is
 the compact $n$-manifold obtained by deleting an open tubular
neighbourhood of $S$. 
Then $\partial X=S_-\sqcup S_+$ and $X$ is the union of $S_-\times[0,1]$ and some handles. 
Given a primitive class $\phi\in H^1(M;{\mathbb Z}),$
let $h(\phi)$ be the minimum number of handles for such $S$ dual to  $\phi$. 
For each $k\ge0$  there is an open subset
$U\subset S(M),$ such that $h(\phi)\le k$ if and only if $[\phi]\in U$.
\end{corollary}

In Section \ref{sec:tropical rank}, we define an analogous complexity function, called the \emph{tropical rank}, for a finitely presented group $G$ by considering all manifolds with fundamental group isomorphic to $G.$ We then re-interpret this function using HNN decompositions of $G$ and relate it to the geometric invariant $\Sigma(G)$ due to Bieri, Neumann and Strebel. To state the properties of this function, let $S(G)$ be the set of equivalence classes of non-trivial homomorphisms $\varphi \co G \to \R,$ where two such homomorphisms are equivalent if they are positive scalar multiples of each other.
If $\varphi(G)$ is discrete, 
then $\rank[\varphi]$ is {\em roughly} the minimum number of generators that must be added
to the amalgamating subgroup $A$ to obtain $B$
 in an HNN extension $G\cong B*_A$ corresponding to $\varphi$.
 
\begin{theorem}\label{thm:intro groups result}
Let $G$ be a finitely presented group. The tropical rank 
$$\rank: S(G) \to \Nzero$$ 
is upper semi-continuous and has image bounded above by the minimal number $n$ of generators of $G.$ Moreover, for rational classes $[\varphi] \in S\Q(G)$ we have
\begin{enumerate}
\item $\rank [\varphi] \le n-1,$
\item $\rank [\varphi] = 0$ if and only if $[\varphi] \in \Sigma(G),$ and
\item $\rank [\varphi] = 0$ and $\rank [-\varphi] = 0$ if and only if $\ker \varphi$ is finitely generated.
\end{enumerate}
\end{theorem}

Most of the results in this paper are known. However we have collected those results that seem most relevant to low dimensional topology and geometric group theory with a view to providing a concise source. Our focus is on {\em handles} in manifolds and {\em generators} for groups
rather than Novikov homology, and is more in the spirit of the Thurston norm and BNS invariant.
We end this introduction by discussing open questions about the special case of 3--manifolds.

Suppose $M$ is a closed, orientable 3--manifold.
There is a natural identification $S(M) = S(\pi_1(M)).$ The first Morse rank on $S(M)$ gives an upper bound on the tropical rank on $S(\pi_1(M)),$ and we do not know if this is always an equality. Thurston~\cite{Thu1986} defined a semi-norm on $H^1(M; \R)$. 
Each primitive class $c \in H^1(M; \Z)$ can be represented by an embedded, orientable surface $F \subset M.$ If $F$ is connected, define $\chi_-(F) = \max \{ 0, -\chi(F)\},$ and otherwise $\chi_-(F) = \sum \chi_-(F_i),$ where the sum ranges over its connected components. Then the semi-norm of $c$ is the infimum of $\chi_-(F)$ taken over all surfaces $F$ representing $c.$ Thurston showed that the unit ball of this semi-norm is a finite sided polytope, and that the complexity zero classes correspond to the interiors of top-dimensional faces of this polytope, called the \emph{fibered faces}. In particular, the set of cohomology classes of non-singular integral classes has image that is dense in an open \emph{polyhedral} subset of $S(M).$ Apart from this, we do not know how the first Morse rank and the tropical rank are related to the polyhedral structure of the Thurston norm ball. 


\section{Manifolds}
\label{sec:Morse rank for manifold}

Let $M$ be a closed \smooth $n$--manifold. 
We show that elements of $H^1(M;{\mathbb R})$ 
are represented by Morse 1--forms.
This is used
to  define a complexity taking values in $\Nzero^{n+1}$ (using lexicograhic order)
 for a first cohomology class.  
Nearby cohomology classes are represented
by nearby Morse 1--forms with the {\em same} singularities. This implies the
set of classes with complexity less than some value is open.
We then relate this complexity to handle decompositions of $M$ and the Novikov numbers of $M.$ 
The notation and terminology in this section follows Farber \cite{Farber}, 
and most of the results are surely known to the experts.


\subsection{The Morse complexity}

Let $M$ be a closed, \smooth $n$--manifold. 
A smooth function $f:M\longrightarrow {\mathbb R}$  is {\em Morse} if
at every critical point of $f$ the Hessian $\dd^2f$ is non-singular. This Hessian is a quadratic
form 
$$-(x_1^2+\cdots + x_i^2)+(x_{i+1}^2+\cdots+x_n^2)$$ 
of {\em index} $i$.
A Morse function
gives a handle decomposition of $M$ with $i$--handles corresponding to critical points of index $i$.

Let $\omega$ be a closed 1--form on $M.$ 
For every open ball $U \subset M,$ there is a smooth function $f_U \co U \to \R$ satisfying $\omega |_U = \dd f_U$ and $f_U$ is determined uniquely by $\omega$ up to a constant. In local coordinates, we have:
$$
\omega_x = \sum a_i(x) \dx_i
$$
and hence $a_i(x) = \frac{\partial f}{\partial x_i}(x).$ It follows that the zeros of $\omega$ in $U$ are precisely the critical points of $f_U.$ A zero of $\omega$ is called \emph{non-degenerate} if the critical point of the corresponding function is a Morse singularity, and the index of the singularity is termed the \emph{Morse index} of the zero of $\omega.$ The form $\omega$ is termed \emph{Morse} if all of its zeros are non-degenerate. 

For a Morse 1--form $\omega$ denote by $m_i=m_i(\omega)$ the number of zeros of Morse index $i.$ 
This is finite since $M$ is compact. This gives a {\em complexity vector}
 $\widetilde{m}(\omega)=(m_0,\cdots,m_{n})$. We use lexicographic ordering on these vectors
 so $c< d$ if there is $i\ge 0$ with $c_i<d_i$ and $c_j=d_j$ for all $j<i$.
 
 By Theorem~\ref{lem:homologous to Morse} every element   $\xi \in H^1(M; \R)$ is represented by a 
Morse 1--form $\omega$ and we define $\widetilde{m}(\xi)$ to be the minimum of $\widetilde{m}(\omega)$ over such $\omega$. Note that 
$\widetilde{m}(\xi)=\widetilde{m}(r\xi)$ for $r>0$, however in general $\widetilde{m}(\xi)\ne \widetilde{m}(-\xi)$, so $\widetilde{m}$ decends to a well-defined function on the sphere,
$$
m:S(M) = (H^1(M; \R) \setminus \{ 0\}) / \R_+\longrightarrow\Nzero^{n+1},
$$
called the \emph{Morse complexity function}. Note that $S(M) \cong S^{k-1},$ where $k = \rank H^1(M; \R)$.

In the case $\xi=0$ a minimising Morse 1--form  is $\omega=\dd f$ for some Morse function $f$
  on $M$.  Then $\widetilde{m}(0)$ gives the least number of handles in a decomposition of $M$ using
  lexicographic ordering. In particular $\widetilde{m}(0)=(1,0,\cdots,0,1)$ if and only if $M$ is
 {\em diffeomorphic} to a sphere. The existence of exotic $7$--spheres implies
 that $\widetilde{m}(0)$ is not an invariant of homotopy type.

Instead of minimising the whole complexity vector, one may also be interested in minimising any one of its component functions. For $\xi \in H^1(M; \R),$ the \emph{$i$--th Morse rank} is
$$\mrank_i(\xi) = \min_{[\omega] = \xi}\; m_i(\omega).$$
We again have $\mrank_i(\xi)=\mrank_i(r\xi)$ for $r>0$, however in general $\mrank_i(\xi)\ne \mrank_i(-\xi)$. We therefore view $\mrank_i$ as an integer-valued function on the sphere $S(M),$ and note that in general, the $i$--th Morse rank is not the $i^{th}$ component of the Morse complexity function.
  
\begin{proposition}\label{pro:morse rank for manifold}
The $i$--th Morse rank $$\mrank_i\co S(M) \to \Nzero$$ is upper semi-continuous for each $i \in \{0,1,\ldots, n\}.$ That is, for each $x \in S(M),$ there is an open neighbourhood $N(x)$ with respect to the usual topology of $S^{k-1}$ such that $\mrank_i(y) \le \mrank_i(x)$ for all $y \in N(x).$
\end{proposition}

Since $S(M)$ is compact, it follows that the range of $\mrank_i\co S(M) \to \Nzero$ is finite. Since the function is upper semi-continuous, we have the following immediate consequence:

\begin{corollary}\label{cor:minimal points}
The set $\Sigma_i$ of all points of minimal $i$--th Morse rank is open in $S(M).$ Moreover, rational points are dense in $\Sigma_i.$
\end{corollary}

\begin{remark}
For all rational points, we have $\mrank_0[\xi] = 0 = \mrank_n[\xi]$ since they can be represented by circle valued Morse functions. However, Corollary \ref{cor:minimal points} does not imply that
$$\mrank_0 \equiv 0 \equiv \mrank_n.$$
Are there examples, where the functions $\mrank_0$ and $\mrank_n$ are not trivial?
\end{remark}

\begin{example}[(The torus)]\label{exa:torus}
Consider the torus $T^k=V/{\mathbb Z}^k,$ where $V={\mathbb R}^k$.
A linear map $L\in V^*$ gives an exact 1--form $\dd L$ on ${\mathbb R}^k,$ which
is preserved by translations, and therefore is the pullback of a closed 1--form $\omega_L$ on $T^k$. Since there are no singularities, $\mrank_1$ vanishes identically.
\end{example}

\begin{example}
There are manifolds $M$ with the property that $\mrank_1\co S(M)\to \Nzero$ does not vanish identically but is zero on a dense set; for instance 3--manifolds with each face of the Thurston norm ball a fibred face, such as (zero-framed surgery on) the complements of the Whitehead link or the Borromean rings (see \cite{Thu1986}).
\end{example}

\begin{example}[(Surfaces)]
Let $F$ be a surface of genus $g\ge 1,$ and $\xi \in H^1(F, \Z)$ be a non-trivial primitive element. Then $\xi$ is dual to a circle valued Morse function $F \to S^1$ with no critical points of indices 0 or 2. Since critical points of index 1 can only be cancelled against those of index 0 or 2, the Euler characteristic of $F$ implies $m(\xi) = (0, 2g-2, 0)$ and $\mrank_1(\xi) = 2g-2.$
\end{example}


\subsection{Proofs of Theorem~\ref{thm:intro main} and Proposition~\ref{pro:morse rank for manifold}}

Endow $M$ with a Riemannian metric. We let 
$$B(x;r) = \{ y \in M \mid d(x,y) < r\}.$$
If $B(x;r)$ is called an \emph{open ball}, then it is understood that $r$ is chosen such that $B(x;r)$ is diffeomorphic to $\Int B^n \subset \R^n.$ This condition holds for $r$ sufficiently small. At each $x \in M,$ we have a well-defined real number, 
$$|| \ \omega \ ||_x = \sqrt{\langle \omega_x, \omega_x\rangle_x},$$ 
arising from the inner product at $x.$ We have $|| \ \omega \ ||_x = 0$ if and only if $x$ is a zero of $\omega.$ 
The {\em norm} of $\omega$ is $||\omega||=\sup ||\omega||_x$.
A closed 1--form $\omega$  defines a de-Rham cohomology class $[\omega]\in H^1(M;{\mathbb R})$.

It is well known that the set of Morse functions is dense in the set of smooth functions.
Lalonde, McDuff and Polterovich \cite[Lemma 5.1]{LMcP} show that the harmonic representative  of a
 non-trivial  class in $H^1(M,{\mathbb R})$ is 
a Morse 1--form  for a generic Riemannian metric. 
We give an elementary proof of
the density of Morse forms in every cohomology class. Another proof is in \cite[Ch.\thinspace 2 Theorem 1.25]{Pajitnov}.

\begin{theorem}\label{lem:homologous to Morse} If $M$ is a closed Riemannian manifold, then the
set of Morse 1--forms in each cohomology class is dense.\end{theorem}
\begin{proof} We must show that given $\epsilon >0$ and a closed 1--form $\omega,$ there is a Morse 1--form, $\eta,$ 
such that $[\eta]=[\omega]\in H^1(M;{\mathbb R})$ and $||\omega-\eta||<\epsilon$.

Suppose $\psi$ is a closed \smooth 1--form on $M$ and 
$U\subset M$ is an open set such that $\psi|\overline U$ is Morse, and $\psi|\overline U$ 
has finitely many zeros which are all in $U$. Suppose
$B:=B(x;r)\subset M$ is an open ball contained in larger open balls $2B:=B(x;2r)$ and $3B=B(x;3r)$. 
Given $\delta>0$ 
we show below that  there is a \smooth 1--form $\psi_1$ on $M$ such
  that $\psi=\psi_1$ outside $3B$ and $||\psi_1-\psi||<\delta$ and 
 $\psi_i|\cls(U\cup B)$ is Morse with finitely many singularities, which are all in $U\cup B$.
Since
$\psi_1$ and $\psi$  are equal outside the contractible set $3B$ it follows they
represent the same cohomology class.

 Assuming this, there is a finite cover of $M$ by balls $B_i:=B(x_i;r)\subset M$ with $1\le i\le m,$ each contained in
 balls $B(x_i;2r)\subset B(x_i;3r)$. We start with $\psi_0=\omega$ and $U_0=\emptyset$ and $\delta=\epsilon/m$ 
 and inductively apply the above to produce
 a  1--form $\psi_k$ which is Morse on $U_k=\cup_{i=1}^k B_i$. Then $\eta=\psi_m$ is as required.

It remains to prove the claim. Since $\psi$ has finitely many zeros in $U,$ there is an open set $V$ with
 $\overline{V}\subset U$ and $V$ contains all the zeros of $\psi|\overline U$.  Set
   $\mu=\inf ||\psi ||_x$ where the infimum is over $x$ in the compact set $\overline U\setminus V$. Then $\mu>0$.
 
The closed sets $C=\cls(B\setminus U)$ and $D=\cls(M\setminus 2B)\cup\cls(V)$ are disjoint.
 Thus there is a smooth function
 $\lambda:M\longrightarrow [0,1]$ with $\lambda(C)=1$ and $\lambda(D)=0$. 
 Let $K=1+ ||\dd\lambda||$.
 
Since $\psi$ is closed and $3B$ is a ball, integrating $\psi$ along paths starting at the base point gives
 a smooth function $g:3B\longrightarrow{\mathbb R}$ with $\dd g =\psi$ on $3B$.
 Since Morse functions are dense in $C^\infty(\Int(B^n)),$ there is a Morse function
$h:3B\longrightarrow{\mathbb R}$ with $\sup |g-h|<\mu/3K$ and $|| \dd(g-h)|| <\mu/3$. 

The function $f:3B\longrightarrow{\mathbb R}$ given by
$f=g+\lambda\cdot (h-g)$ equals $h$
on $C$ and $g$ on $3B\cap(V\cup(3B\setminus 2B))\subset D$.
 Thus  $f$ is Morse on $3B\cap(C\cup V)$. We have
 $$\dd f=\dd g +  (h-g)\dd\lambda + \lambda\dd(h-g)$$
 There are no critical
points of $f$ in $3B\cap (\overline U\setminus V)$ since on this set
$  ||\dd f||>\mu/ 3$
because $||\dd g||\ge\mu$ and $||\dd\lambda||\le K$ so $||(h-g)\dd\lambda||<\mu/3$
and $||\lambda\dd(h-g)||\le ||\dd(h-g)||<\mu/3$. 

Putting this together, $f$ is Morse on the subset of $3B$ contained in $C\cup V\cup (\overline{U}\setminus V)= C\cup\overline{U}=B\cup\overline{U}$. 

Observe that
$\psi$ and $\dd f$ are equal on  $3B\setminus 2B$. Thus we may define $\psi_1$ to be $\psi$ on $M\setminus 2B$
and $\dd f$ on $3B$. Then $\psi_1$ is Morse on 
$3B\cap(B\cup\overline{U})$ because
 $\psi_1=\dd f$ there. It is Morse on $\overline{U}\setminus 2B$ because it is $\psi$ there. Hence $\psi_1$ is 
Morse on $\overline{U}\cup B$ as required.
\end{proof}

The next result implies that nearby homology classes are represented by nearby 1--forms.
\begin{lemma}\label{nearbyclass} If $M$ is a closed Riemannian manifold and $\epsilon>0,$ there is a neighbourhood $U\subset H^1(M;{\mathbb R})$
of $0$ with the property that for every $\phi\in U,$ there is a \smooth closed 1--form $\eta$ with $[\eta]=\phi$ and $||\eta||<\epsilon$.
\end{lemma}
\begin{proof} 
We first prove the result for the $k$--torus $T^k=V/{\mathbb Z}^k,$ where $V={\mathbb R}^k$.
As noted in Example~\ref{exa:torus}, a linear map $L\in V^*$ gives an exact 1--form $\dd L$ on ${\mathbb R}^k,$ which
is preserved by translations, and therefore is the pullback of a closed 1--form $\omega_L$ on $T^k$.
The map $V^*\longrightarrow \Omega^1(T^k)$ given by $L\mapsto \omega_L$ is continuous and 
$V^*\longrightarrow H^1(T^k;{\mathbb R})$ given by $L\mapsto[\omega_L]$ is an
isomorphism.

For the general case let $k=\beta_1(M).$ There is a \smooth map $f:M\longrightarrow T^k,$ which induces
an isomorphism $f^*:H^1(T^k;{\mathbb R})\longrightarrow H^1(M;{\mathbb R})$. Then $f^*(\omega_L)$
is a closed \smooth 1--form  on $M,$ which varies continuously with $L$ and this gives an isomorphism $V^*\longrightarrow H^1(M;{\mathbb R})$.
\end{proof}

Next is a local stability result: nearby cohomology classes are 
represented by Morse 1--forms with the {\em same} zeros.

\begin{lemma} \label{lem:local stability}
Given a closed Morse 1--form $\omega$ and $\delta>0$,
 there is a neighbourhood $U\subset H^1(M;{\mathbb R})$ of 
$0$ with the property that
if $\alpha\in U,$ 
  then there exists a Morse 1--form $\eta$ with $[\eta]=\alpha+[\omega]$  and $||\omega-\eta||<\delta$ and there
  is an open set $A\subset M,$ which contains all the
  zeros of $\eta$ and of $\omega$ and $\omega|A=\eta|A$. In particular  $\widetilde{m}(\eta) = \widetilde{m} (\omega)$.
\end{lemma}

\begin{proof} 
Let $\omega$ be a Morse 1--form. Then $\omega$ has a finite number of zeros, say $z_1,...,z_m.$ For each $z_i$ choose a radius $r_i > 0$ such that 
$B_i=B(z_i; r_i)$, $2B_i= B(z_i; 2 r_i)$ and $3B_i=B(z_i;3r_i)$ are open balls. 
We may do this so the balls $3B_i$ are pairwise disjoint. Define $A = \bigcup B_i$ and $3A=\cup 3B_i$ and set 
$$\delta= \inf \{ || \omega ||_x \ : \ x\in M\setminus A\}.$$
Then $\delta>0$ because  $M\setminus A$ is compact.
Choose a smooth function $\lambda_i:M\longrightarrow [0,1]$ with $\lambda_i(B_i)=1$ and 
$\lambda_i(M\setminus 2B_i)=0$.
Choose $\epsilon>0$ smaller than $\delta/3$ and $(1/3)(1+||\dd\lambda_i||)^{-1}$ for every $i$.

Let $U\subset H^1(M;{\mathbb R})$ be the neighbourhood of $0$ given by Lemma~\ref{nearbyclass}. 
Given $\alpha\in U$ there is a closed 1--form $\phi$ representing $\alpha$ with $||\phi||<\epsilon$.
By Theorem~\ref{lem:homologous to Morse} we may assume $\phi$ is Morse.

Since $||\phi||<\delta$ all the zeros of $\eta=\omega+\phi$ are contained in $A$. 
Below we construct a Morse 1--form
$\eta'$ which equals $\eta$ in $M\setminus 3A$ so $[\eta']=[\eta]=\alpha+[\omega].$ Moreover
$\eta'=\omega$ on $A$
completing the proof.

To construct $\eta'$ we modify $\eta$ in $B=B_i\subset 3B_i=3B$ as follows. 
There are smooth functions $f,g: 3B\longrightarrow{\mathbb R}$ with $\omega=\dd f$ and $\phi=\dd g$ on $3B$. Set
$$h= f +\lambda_i(g-f)$$
Then $h=g$ on $B$ and $h=f$ on $3B\setminus 2B$. A calculation as in the proof of
 Theorem~\ref{lem:homologous to Morse} shows $h$ is Morse and with one singular point at $z_i$. 
Define $\eta'$ to be $\eta$ in $M\setminus 2B$ and  $\dd h$ in $3B$. On the overlap $3B\setminus 2B$ these
are both $\eta$, and  $\eta=\omega$ on $B$.
Since the balls $3B_i$ are pairwise disjoint these modifications can be done independently.
 \end{proof}

\begin{proof}[Proofs of Theorem~\ref{thm:intro main} and Proposition \ref{pro:morse rank for manifold}]
The statements follows directly from the local stability result (Lemma~\ref{lem:local stability}) together with the 
following two facts. First, the topology on $S(M)$ induced by the norm 
agrees with the usual topology. Second, for each class $\xi \in H^1(M; \R),$ 
one may choose a representative $\omega_1$ with the property that $\widetilde{m}(\omega_1) = m(\xi),$ and a representative $\omega_2$ with the property that $\mrank_i(\omega_2) = \mrank_i(\xi).$
\end{proof}


\subsection{The $i$--th handle rank of a rational class}

There is a topological description of the Morse complexity and the Morse ranks for rational classes. Suppose $(N,\partial_-N,\partial_+N)$ is a {\em cobordism}. Thus $N$ is a \smooth compact connected $n$-dimensional manifold  with $\partial N=\partial_-N\sqcup\partial_+N$. 
Suppose ${\mathcal H}$ is a handle decomposition of $N$ given by attaching a finite number of handles to a collar $\partial N_-\times[0,1]$ of $\partial_-N$. 
Such a handle decomposition determines a vector 
$$\hat h(N,\partial_-N,{\mathcal H})=(h_0,h_1,\cdots,h_n)\in \Nzero^{n+1},$$ 
where $h_i$ is the number of handles of index $i$.
The {\em handle complexity} of $N$ rel $\partial_-N$ is the vector 
$$\tilde h(N,\partial_-N) = \min_{\mathcal H}  \; \hat h(N,\partial_-N,{\mathcal H}),$$ 
where the minimum is taken with respect to lexicographical ordering over all such handle decompositions. Thus $\tilde  h(N,\partial_-N)=0$ if and only if $N$ is diffeomorphic to $\partial N_-\times[0,1]$.

Similarly, the \emph{$i^{th}$ handle rank} of $N$ rel $\partial_-N$ is 
\begin{equation}\label{eq:i-th handle number}
\hrank_i(N,\partial_-N) = \min_{\mathcal H} \; h_i.
\end{equation}
 
 A primitive element  $\beta\in H_{n-1}(M;{\mathbb Z})$ is represented by a compact, connected, \smooth 2--sided submanifold $V\subset M$ with tubular neighbourhood $\nu(V)$. 
 An orientation of $M$
determines a {\em transverse orientation} to $V$ and hence a cobordism $(N=M\setminus\nu(V),V_-,V_+)$. The minimum of $\tilde h(N, V_-)$ over such $V$ is the {\em handle complexity} $h(\beta)$ of $\beta$, and the minimum of $\hrank_i(N, V_-)$ is the \emph{$i^{th}$ handle rank} $\hrank_i(\beta)$ of $\beta.$
In general $h(-\beta)\ne h(\beta)$ so the choice of orientation on $M$ is necessary.

Finally, if $\xi \in H^1(M; \Z)$ is primitive, denote the Poincar\'e dual to $\xi$ by $D_M(\xi) \in H_{n-1}(M;{\mathbb Z}).$ Then define 
$$h(\xi) = h(D_M(\xi)) \qquad\text{and}\qquad \hrank_i(\xi) = \hrank_i(D_M(\xi))$$
to be the handle complexity  and $i^{th}$ handle rank of $\xi$ respectively.

\begin{lemma}\label{lem:hrank=mrank} Suppose $M$ is a smooth closed oriented manifold.
For each primitive class $\xi\in H^1(M; \Z),$ 
we have $h(\xi) = m(\xi)$ and $\hrank_i(\xi) = \mrank_i (\xi).$
\end{lemma}

\begin{proof} 
This follows from the well-known correspondence between Morse functions on $M$ and cobordisms associated to regular level sets (see \cite{Mil}, \cite{Mil1965}).
\end{proof}

\begin{remark}
The group rank, which is defined below, gives a lower bound on the first handle rank, and hence on the first Morse rank (see Remark~\ref{rem:grank lower bounds hrank}).
\end{remark}


\subsection{Lower bounds from Novikov numbers}

For every first cohomology class $\xi \in H^1(M; \R)$ and every $k\in \Nzero,$ there is a \emph{Novikov Betti number} $b_k(\xi)$ and a \emph{Novikov torsion number} $q_k(\xi)$ (see \cite{Farber}, \S1.5 for three equivalent definitions). For $\xi=0,$ $b_k(\xi)$ is the usual $k^{th}$ Betti number of $M$ and $q_k(\xi)$ is the minimal number of generators of the torsion subgroup of $H_k(M; \Z).$ For general $\xi,$  the ring $\Z$ is replaced with a local system of modules over a fancier ring. The Novikov inequalities (see \cite{Farber}, \S2.3) state that
\begin{equation}
m_k(\omega) \ge b_k(\xi) + q_k(\xi) + q_{k-1}(\xi),
\end{equation}
if $[\omega] = \xi.$ Moreover, $b_k(\xi) =  b_k(r\xi)$ for every $0 \neq r \in \R$ and $q_k(\xi) =  q_k(r\xi)$ for every $r>0.$ It follows that for $[\xi] \in S(M),$ we have
\begin{equation}
\mrank_k[\xi] \ge b_k(\xi) + q_k(\xi) + q_{k-1}(\xi).
\end{equation}
Moreover, $b_0(\xi) =  0 =q_0(\xi),$ so the relationship is particularly nice when $k=1$:
\begin{equation}
\mrank_1[\xi] \ge b_1(\xi) + q_1(\xi).
\end{equation}
We call a function $f\co S(M) \to \Nzero$ \emph{polyhedral} if each superlevel set $f^{-1}[k,\infty)$ is a spherical polytope.
Farber (\cite{Farber}, \S1.6) gives a nice description of the geometry of the Novikov inequalities and, in particular, a polyhedrality result. Since the Morse rank has this polyhedral lower bound, this triggers the following question:
\begin{question}
Is the Morse rank a polyhedral function?
\end{question}

\begin{example}[(3--manifolds with arbitrarily large Morse rank)] Farber (\cite{Farber}, \S3.4.2) gives a family of closed 3--manifolds $X_n,$ which are obtained as the connected sum of $S^2\times S^1$ and the 0--surgery on the connected sum of $n$ trefoil knots, and satisfy $H^1(X_n; \Z) \cong \Z^2.$ Moreover, the generator $\xi$ supported by the  $S^2\times S^1$ summand has the property that every Morse closed 1--form representing it has at least $n$ zeros of index 1. This is established using inequalities arising from a ring, which is more elaborate than the Novikov ring (see \cite{Farber}, Theorem 3.6).
\end{example}


\section{Groups}
\label{sec:tropical rank}


Let $G$ be a finitely presented group, and $S(G)$ be the set of equivalence classes of non-trivial homomorphisms $\varphi \co G \to \R,$ where two such homomorphisms are equivalent if they are positive scalar multiples of each other. Then $S(G) \cong S^{n-1},$ where $G / G' \cong \Z^n \oplus$torsion.
Define the \emph{tropical rank} of the class $[\varphi] \in S(G)$ by
$$
\rank [\varphi] = \min_M \{ \mrank_1 [\xi] \},
$$
where the minimum is taken over all smooth closed manifolds $M$ with $\pi_1(M)=G,$ and where $\xi$ is the image of $\varphi$ under the canonical identification 
$$\Hom(\pi_1(M), \R) \cong H^1(M;\R).$$ 
The tropical rank is a well-defined function $S(G) \to \Nzero,$ since $r \xi$ is the image of $r \varphi$  for $r>0$ and since $\mrank_1$ is integer valued. The definition involves taking the minimum, and so Proposition~\ref{pro:morse rank for manifold} directly implies:
\begin{proposition}
The tropical rank is upper semi-continuous and has bounded image.
\end{proposition}

The word \emph{tropical} alludes to the fact that this rank gives information on the complement of the Bieri-Neumann-Strebel invariant \cite{BNS}---this complement was characterised by Brown \cite{Brown} using group theoretic valuations, and hence can be viewed as a tropical set.

The above proposition gives the first part of Theorem~\ref{thm:intro groups result}. The remaining parts concern rational classes and follow from the equality of Morse rank and handle rank (Lemma~\ref{lem:hrank=mrank}) together with the equality of handle rank and a \emph{group rank} defined using HNN extensions (Lemma~\ref{lem:hrank=grank}), fundamental facts from \cite{BNS} (summarised in Lemma~\ref{lem:g-rank and BNS}) and an application of Magnus rewriting (Lemma~\ref{lem:upper bound on group rank}).


\subsection{The Bieri-Neumann-Strebel invariant and HNN extensions}

We now define the Bieri-Neumann-Strebel invariant $\Sigma(G)$. Choose a finite generating set for $G$ and denote $\mathcal{G}$ the corresponding Cayley graph, with the convention that $G$ acts from the left on this graph. Given the non-trivial homomorphism $\varphi \co G \to \R,$ define a $G$--equivariant map $\tilde{\varphi}\co \mathcal{G} \to \R$ by sending vertices to their images under $\varphi$ and extending linearly over edges. Denote $\mathcal{G}_\varphi$ the maximal subgraph of $\mathcal{G}$ contained in $\tilde{\varphi}^{-1}(-\infty, 0].$ It turns out that the connectedness of $\mathcal{G}_\varphi$ is independent of the representative of $[\varphi]$ and the generating set for $G.$ Define $\Sigma(G) \subseteq S(G)$ to be the set of precisely those $[\varphi]$ for which $\mathcal{G}_\varphi$ is connected.\footnote{The definition of the BNS invariant involves various conventions and choices of sign, and one often encounters a definition using the preimage of $[0, \infty)$ instead of $(-\infty, 0].$ The above definition matches \cite{BNS, FT15}.} It is shown in \cite{BNS} that $\Sigma(G)$ is an open subset of $S(G).$ 

If the non-trivial homomorphism $\varphi \co G \to \R$ has discrete (and hence cyclic) image, then it is a positive scalar multiple of a unique epimorphism $G \twoheadrightarrow \Z,$ and it is termed a \emph{discrete} homomorphism. The point $[\varphi] \in S(G)$ is called \emph{rational.} The set 
$$
S\Q (G) = \{ [\varphi] \in S(G) \mid \varphi \text{ is discrete}\} 
$$
of rational points is dense in $S(G).$ 

Suppose $G$ is a finitely presented group with epimorphism $\varphi\co G \twoheadrightarrow \Z.$ 
An {\em associated HNN--extension of $(G,\varphi)$}
is $(B,A,t,\alpha),$ where:
\begin{equation}\label{eq:HNN extension}
G = \langle\; t, B \mid A  = t^{-1}\alpha (A)t\;\rangle,
\end{equation}
and $A \subseteq B \subseteq \ker \varphi,$ both $A$ and $B$ are finitely generated, $\alpha \co A \to B$ is a monomorphism, and $t \in G$ with $\varphi(t) =1.$ (A proof of this fact can be found in \cite{BS1978}; we give an independent proof in \S\ref{subsec:group handle-rank}.) In this case, $(G, \varphi)$ is said to \emph{split over $A.$}
If $A=B,$ then the HNN-extension is called \emph{ascending} in \cite{BNS} (see the discussion in \S4.1 of \cite{FT15}). The following two facts are established in \cite{BNS}:
\begin{enumerate}
\item $[\varphi] \in \Sigma(G)$ if and only if $\varphi$ corresponds to an ascending HNN extension;
\item $\ker \varphi$ is finitely generated if and only if $[\varphi] \in \Sigma(G)$ and $[-\varphi] \in \Sigma(G).$
\end{enumerate}

In \cite{FT15}, the complexity of $[\varphi]$ is defined as the minimal rank of the group $A,$ where the minimum is taken over all HNN--extensions of $G$ of the form (\ref{eq:HNN extension}), and this gives a measure analogous to the Thurston norm~\cite{Thu1986}. We will show that the tropical rank defines a complementary complexity, namely the \emph{minimal difference} between $A$ and $B.$ This will be established through two alternative viewpoints---using HNN extensions and handle decompositions.


\subsection{The group-rank}

Suppose $G$ is a finitely presented group with epimorphism $\varphi\co G \twoheadrightarrow \Z.$
For an associated HNN extension (\ref{eq:HNN extension}), define $\grank (B, A)$ to be the minimal number of elements one needs to add to $A$ in order to generate B, i.e.
$$\grank(B, A) = \min_{B = \langle b_1,\ldots b_n, A\rangle} n.$$
Then define
\begin{equation}\label{eq:group rank}
\grank(\varphi) = \min (\ \grank (B, A)\ ),
\end{equation}
where the minimum is taken over all HNN--extensions of $G$ of the form (\ref{eq:HNN extension}).

The above stated facts from \cite{BNS} concerning the Bieri-Neumann-Strebel invariant give the following characterisation of classes with trivial group rank:
\begin{lemma}\label{lem:g-rank and BNS}
We have:
\begin{enumerate}
\item $\grank (\varphi) = 0$ if and only if $[\varphi] \in \Sigma(G),$ and
\item $\grank (\varphi) = 0$ and $\grank (-\varphi) = 0$ if and only if $\ker \varphi$ is finitely generated.
\end{enumerate}
\end{lemma}


\subsection{The handle-rank}
\label{subsec:group handle-rank}

Define
\begin{equation}\label{def:hrank}
\hrank (\varphi) = \min_M \{ \hrank_1 (\xi) \},
\end{equation}
where the minimum is taken over all closed \smooth $n$--manifolds $M$ with $\pi_1(M) = G,$ and where $\xi$ is the image of $\varphi$ under the canonical identification $\Hom(\pi_1(M), \Z) \cong H^1(M;\Z).$

We wish to compare the handle rank with the group rank. Note that the groups $A$ and $B$ in presentation (\ref{eq:HNN extension}) are not necessarily finitely presentable (according to the Higman embedding theorem they are recursively presentable). We work around this issue using \emph{``fake HNN--extensions"} as follows.

Suppose $G$ is a finitely presented group and $\varphi \co G \twoheadrightarrow \Z.$ There is a finite presentation of $G$ of the form
\begin{equation}\label{eq:fake HNN}
\langle\; s, d_i, c_j \mid r_k(d_1, \ldots ,c_1, \ldots)=1,\; s c_j s^{-1} = w_j(d_1, \ldots ,c_1, \ldots) \;\rangle
\end{equation}
with $\varphi(s)=1$ and $\varphi(d_i) = \varphi(c_j) =0.$ The existence of this presentation can be seen either algebraically or topologically. 

Algebraically, one can apply Magnus re-writing to any presentation of $G$ to convert it to the desired form (see \cite[\S IV.5]{LS1977} and also the proof of Lemma~\ref{lem:upper bound on group rank}).

The topological argument starts by choosing a closed, connected \smooth manifold $M$ with $\pi_1(M) = G.$ Then $\varphi \co G \twoheadrightarrow \Z$ determines a \smooth map $f\co M \to S^1.$ Let $p\in S^1$ be a regular value with connected level set $H = f^{-1}(p),$ and $N = M \setminus \nu(H)$ be the complement of an open regular neighbourhood of $H$ in $M.$ Then $\partial N = H_+ \cup H_-$ is the disjoint union of two copies of $H$ and there are two (not necessarily injective) inclusion homomorphisms $\alpha_\pm \co \pi_1(H) \to \pi_1(N).$ Applying the Generalised Van Kampen Theorem \cite[Theorem 6.2.11]{Geo} to $M \setminus H$ and $\nu(H)$ gives a presentation of $\pi_1(M) = G$ of the form
$$\langle\; s, \pi_1(N) \mid s\alpha_+(h)s^{-1} = \alpha_-(h)\; \ \forall h \in \pi_1(H) \;\rangle,$$
where $\pi_1(N)$ and $\pi_1(H)$ are finitely presented groups since both $N$ and $H$ are compact.
Since $\pi_1(H)$ is finitely generated  we only need the relations $s\alpha_+(h)s^{-1} = \alpha_-(h)$ for a generating set. 
This gives a presentation in the form (\ref{eq:fake HNN}) as follows. Denote the generators and relators of $\pi_1(N)$ by $d_i$ and $r_k$ respectively and for each relation of the form $s\alpha_+(h)s^{-1} = \alpha_-(h)$ add one generator $c_j$ and the relations 
$c_j = \alpha_+(h)$ and $sc_js^{-1} = \alpha_-(h)$
with $\alpha_\pm(h)$ expressed as words in the generators of $\pi_1(N).$
This completes the topological argument that a presentation of the form (\ref{eq:fake HNN}) exists. 

From (\ref{eq:fake HNN}), one obtains an HNN extension as follows. Let $A = \langle c_j \rangle \le G$ and $B = \langle d_i, c_j \rangle \le G.$ Then $A$ and $B$ are finitely generated (but possibly not finitely presented) subgroups of $G$ with $A \le B.$ Moreover, $s A s^{-1} \le B$ by construction, so the conjugation map $\alpha: a \mapsto sas^{-1}$ takes $A$ to an isomorphic subgroup of $B.$ Whence $G = \langle s, B \mid A = s^{-1}\alpha(A)s\rangle$ and we also have 
$$
\grank(B,A) \le |\ \{ d_i \} \ |.
$$

\begin{lemma}\label{lem:hrank=grank}
$\hrank(\varphi) = \grank (\varphi)$
\end{lemma}

\begin{proof}
Each relative handle decomposition gives rise to a fake HNN--decomposition of the form (\ref{eq:fake HNN}). For any $\varphi,$ there is a manifold $M$ with relative handle decomposition $(N,H_-, \mathcal{H})$ realising the minimum in (\ref{def:hrank}). Using the above notation with $B = \im(\pi_1(N)\to \pi_1(M))$ and $A = \im(\pi_1(H_-)\to \pi_1(M))$ gives:
$$
\grank(\varphi) \le \grank(B,A) \le |\ \{ d_i \} \ | = \hrank_1(N,H_-) = \hrank(\varphi).
$$

To prove $\hrank(\varphi) \le \grank (\varphi),$ build a 2--complex with fundamental group $G$ as follows. Let $G = \langle t, B \mid t A t^{-1} = \alpha (A)\rangle$ be an HNN--extension realising minimal group rank for $\varphi.$ Choose finite generating sets $\{ a_i \}$ for $A$ and $\{ a_i, b_j \}$ for $B$ such that $|\{ b_j\}| = \grank(\varphi).$ Since $G$ is a finitely presentable group and generated by $\{t, a_i, b_j \},$ there is a finite presentation of $G$ with these generators and of the form 
\begin{equation}\label{eq:fake HNN 00}
G = \langle t, a_i, b_j \mid r_k(a_1,\ldots,b_1,\ldots),\ t a_i t^{-1} = w_i(a_1,\ldots,b_1,\ldots) \rangle.
\end{equation}
Given the finite presentation, we let $C$ be the free group in $\{a_i\}$ and $D$ be the group 
$$\langle a_i, b_j \mid r_k(a_1,\ldots,b_1,\ldots) \rangle.$$

Build a 2--complex, $Y_0,$ with one 0--cell and 1--cells for the \emph{elements} $b_i,$ $a_i$ and $\alpha(a_i),$ and 2--cells
 for the relators in the presentation of $D$ and the extra relations $\alpha(a_i)w_i^{-1},$ where $w_i$
  is a word in the $a_i$ and $b_i$ equal to the element $\alpha(a_i).$ This 2--complex has fundamental group $D.$ 

Then take the sub-complex $X\subset Y_0$ formed by the cells for $C$ and cross it with an interval, giving a 2--complex, $X\times I.$ Let $Y_1 = Y_0 \cup_{X=(X \times \{0\})} (X\times I).$ Then $Y_1$ still has fundamental group $D.$ 

Let $Y_2$ be the 3--complex obtained from $Y_1$ as follows. For each $a_i,$ identify the 1--cell representing $a_i$ in $X\times \{1\} \subset Y_1$ with the 1--cell representing $\alpha(a_i).$ It follows that $Y_2$ has fundamental group $G.$

Embed $Y_2$ in $\R^5$ and embed $\R^5$ in $\R^6$ in the standard way, taking $\R^5 \ni x \mapsto (x,0) \in \R^6.$ Let $Y_3$ be the boundary of a regular neighbourhood of the embedding. This is a 5--manifold and it has a product region, identified with $W \times (0,1)$ for suitable $W,$ corresponding to the subcomplex $X \times (0,1).$ Identify $W = W \times \{\frac{1}{2}\} \subset Y_3.$ 

We claim that there is a handle decomposition of $\overline{Y_3 \setminus W}$ relative to $W_- \subset \partial (\overline{Y_3 \setminus W}),$ $W_- \cong W,$ with exactly one 1--handle for each $b_i.$ To show this, remove the interior of each 1--handle corresponding to each $b_i,$ and denote the result $Z.$ Then
$$
\pi_1(W_-) \to \pi_1(Z) \to \pi_1(Z, W_-) \to \pi_0 =0,
$$
whence $\pi_1(W_-) \cong \pi_1(Z)$ and it follows from \cite[Lemma 6.15]{RS} (and the remark after that lemma), that no additional 1--handles are needed. Moreover, $W_-$ is connected as a consequence of the chosen embedding of $Y_2$ into $\R^6,$ and it is non-separating due to the presentation of the fundamental group. Hence we have shown that $\hrank(\varphi) \le \grank (\varphi).$ 
\end{proof}

\begin{remark}\label{rem:grank lower bounds hrank}
If $\Gamma = \pi_1(M)$ and $\varphi$ corresponds to $\xi,$ then 
$\grank(\varphi) \le \hrank_1 (\xi),$ giving another lower bound for Morse rank.
\end{remark}

\begin{example}
Example~\ref{exa:torus} implies that the tropical rank for the group $\Z^2 = \pi_1(T)$ vanishes identically. For rational classes, this is easy to verify using HNN extensions, since any epimorphism $\Z^2 \to \Z$ is equivalent to the epimorphism $\varphi \co \langle a, b \mid a^{-1}ba = b\rangle \to \Z$ with $\varphi(a)=1$ and $\varphi(b)=0,$ whence $\rank[\varphi]=0.$
\end{example}


\subsection{Consequences for the tropical rank}

Recall that if $\varphi \co G \to \R$ is discrete, then $\varphi(G)$ is infinite cyclic, and hence there is a unique $r \in \R_+$ such that $r\varphi \co G \twoheadrightarrow \Z.$ It follows from Lemmata \ref{lem:hrank=mrank} and \ref{lem:hrank=grank} that:
$$
\rank [\varphi] = \hrank(r\varphi)= \grank(r\varphi).
$$

\begin{lemma}\label{lem:upper bound on group rank}
If $G$ has a presentation with $n$ generators, then $\rank[\varphi] \le n$ for each $[\varphi]\in S(G),$ and
$\rank[\varphi] \le n-1$ for each rational class $[\varphi]\in S\Q(G).$
\end{lemma}

\begin{proof}
We first address the general case. As in the proof of Lemma~\ref{lem:hrank=grank}, we can construct a closed manifold $M$ with a handle decomposition having exactly $n$ 1--handles in its handle decomposition. We can define a Morse 1--form on $M$ dual to this handle decomposition (see \cite{Mil}, \cite{Mil1965}). This has exactly $n$ critical points of index 1. The local stability result Lemma~\ref{lem:local stability} implies that this is a universal bound on the number of critical points of index 1 of every Morse 1--form on $M.$

Next suppose that $[\varphi]\in S\Q(G).$
Given the presentation $G = \langle\; g_i \mid r_j \;\rangle$ with $n$ generators we apply Magnus rewriting as follows. If $\varphi(g_{i_0})=\pm 1$ for the generator $g_{i_0}$ of $G,$ then choose $s=g_{i_0}^{\pm 1}$ and $d_i = g_i s^{\mp \varphi(g_i)}$ for all $i \neq i_0.$ We have $\varphi(s)=1$ and $\varphi(d_i)=0$ for each $d_i.$ Now rewrite the relators using this new generating set $\{s, d_i\}.$ If $s$ appears in the relator $r_k$, then (up to cyclic permutation) there is a substring in $r_k$ of the form $s^{-1} w s,$ where $w$ is a word in the $d_i.$ We introduce a new generator $c_j,$ the relation $sc_js^{-1} = w$ and replace the substring $s^{-1} w s$ by $c_j$ in $r_k.$ Note that $\varphi(c_j)=0.$ Iterating this procedure results in the desired presentation, and the number of generators $d_i$ is bounded above by $n-1.$

Hence assume $\varphi(g_{i})\neq \pm1$ for all generators of $G.$ Since $\varphi$ is an epimorphism, there are two generators $g_{i_0}$ and $g_{i_1}$ such that their respective images are coprime. So there are $p, q \in \Z$ such that $\varphi$ maps $s = g_{i_0}^p g_{i_1}^q$ to one. Introduce this new generator, and replace each $g_i$ by $b_i = g_i s^{- \varphi(g_i)}$. This gives a new generating set $\{s, b_i\}$ with $n+1$ elements. Rewrite all relators in these generators and add the relator $s^{-1} (  b_{i_0}s^{\varphi({g_i}_0)}   )^p (  b_{i_1}s^{\varphi({g_i}_1)}   )^q.$ The latter lets us choose either $c_{j_0} = b_{i_0}$ or $c_{j_1} = b_{i_1},$ and we put $d_{i} = b_{i}$ for all remaining indices. With this we procede as above, again obtaining a generating set $\{s, c_j, d_i\}$ with at most $n-1$ generators $d_i.$
\end{proof}

\begin{question}
Is the tropical rank a polyhedral function?
\end{question}

\begin{remark}
The definition of tropical rank took into account only the 1--handles (\emph{generators}) since in general, the groups arising in an HNN extension are finitely generated but not necessarily finitely presented. For groups with the property that every finitely generated subgroup is finitely presented, one can also take into account the 2--handles (\emph{relators}), giving a finer $\Nzero^2$--valued complexity, for which one can define analogous $\Nzero^2$--valued group and handle ranks, which agree with it on rational classes.
\end{remark}

\subsection*{Acknowledgements}
{The first author is partially supported by NSF grants DMS--0706887, 1065939, 1207068 and  1045292.}
{The second author is partially supported by Australian Research Council grant DP140100158.}


%
%

\bibliographystyle{amsplain}
\bibliography{refs}

\bigskip


\address{Department of Mathematics, University of California Santa Barbara, CA 93106, USA}
\email{cooper@math.ucsb.edu}

\address{School of Mathematics and Statistics, The University of Sydney, NSW 2006, Australia} 
\email{stephan.tillmann@sydney.edu.au} 
\Addresses

\end{document}